\documentclass[a4paper,12pt]{amsart}  

\usepackage[utf8]{inputenc} 
\usepackage[T1]{fontenc}
\usepackage{lmodern}
\usepackage{xcolor}
\usepackage[mathscr]{euscript}
\usepackage{mathtools}

\usepackage{array} 
\usepackage{enumitem}
\usepackage{amssymb}
\usepackage[colorlinks,linkcolor=blue]{hyperref} 
\usepackage{doi}

\newcommand{\Z}{\mathbb{Z}}
\newcommand{\A}{\mathbb{A}}
\newcommand{\C}{\mathbb{C}}
\newcommand{\vp}{\varphi}

\long\def\comment#1\endcomment{}

\renewcommand{\prod}{\Pi}

\newtheorem{thm}{Theorem}[section]

\newtheorem{prop}[thm]{Proposition}
\theoremstyle{definition}
\newtheorem{defn}[thm]{Definition}
\newtheorem{ex}[thm]{Example}
\theoremstyle{remark}
\newtheorem{remark}[thm]{Remark}

\begin{document}
\title[Early proofs of the Nullstellensatz]{Early proofs of Hilbert's Nullstellensatz}

\author{Jan Stevens}
\address{\scriptsize Department of Mathematical Sciences, Chalmers University of
Technology and University of Gothenburg.
SE 412 96 Gothenburg, Sweden}
\email{stevens@chalmers.se}

\begin{abstract}
By Rabinowitsch' trick Hilbert's Nullstellensatz follows from the weak
Nullstellensatz (Rabinowitsch 1929). The weak version can be shown with
elimination theory. Hilbert's original proof  is also based on successive
elimination. Lasker obtained a new proof using  primary decomposition.
We describe these early proofs
and place them in the development  of commutative algebra up to the
appearance of van der Waerden's Moderne Algebra. We also explain Hentzelt's
Nullstellensatz.
\end{abstract}

\thanks{}
\maketitle

\section*{Introduction}
Hilbert's theorem of zeros, or Nullstellensatz as it is usually called, 
states that if a polynomial $f \in P=k[X_1, \ldots, X_n]$,  
where $k$ is an algebraically closed field, vanishes in all common zeros
of an ideal $I\subset P$,  
then $f^r\in I$ for some natural number $r$.  Usually the proof is reduced
to a special case,
the weak Nullstellensatz, that an ideal without zeros is the whole ring, by an
argument due to Rabinowitsch \cite{Rab29}. 
The weak Nullstellensatz follows by elimination.
Hilbert's original proof \cite{Hilb93} is also based on elimination.  A different proof based on
primary decomposition is due to Lasker \cite{L}. We place these proofs
 in the early development  of commutative algebra.

Rabinowitsch's proof \cite{Rab29}
appeared just in time to be included in the second volume of 
van der Waerden's  Moderne Algebra 
\cite{MAII}. This book
can be seen as marking the end of the early period.
It made the subject widely known, and in fact it is still 
a good introduction to the results we discuss
in this paper. Afterwards new proofs of the weak Nullstellensatz appeared with
a totally different flavour, like Zariski's proof, based on his lemma that
 if a finite integral domain over a field $K$ is a field then it is an algebraic
extension of $K$ \cite{Zar}. The most common modern proofs are variations of this proof.

Rabinowitsch's half page paper claims in fact to give a (complete) proof of the
Nullstellensatz and does not use the term weak Nullstellensatz. It refers
to Emmy Noether's paper \cite{HN} for the statement that 
an ideal without zeros is the whole ring, with a footnote that it also follows
from Kronecker's elimination theory. 
Both the Hentzelt-Noether and the Kronecker theory are based on successive
elimination of variables. This is also the technique Hilbert uses in his proof;  
he adapts Kronecker's device to the homogeneous case. In line with his intended
application in invariant theory  Hilbert formulates and proves in \cite{Hilb93} 
the Nullstellensatz for homogeneous polynomials.

The Nullstellensatz was instrumental to the creation of
the concept of primary ideal \cite{L}. Lasker's definition is different
from the modern one, which is due to Emmy Noether \cite{N21}. Macaulay paraphrases
Lasker as follows: if the product of two ideals is contained in a primary ideal, and 
 if one does not contain its zero set the other is
contained in the ideal \cite{M16}.  
To be able to work with this definition it is essential to know 
that a  prime ideal consists of all polynomials vanishing on  its  zero set. 
This is a special case of the Nullstellensatz. It is also possible to show the result
directly and use it in turn to prove the Nullstellensatz. Both Lasker \cite{L-err,L}
and Macaulay \cite{M16} do this. 
Their methods are different. Macaulay uses more specific computations
and he uses Kronecker's theory of the resolvent to describe the zero set 
of ideals. Lasker reasons more abstractly.

The Moderne Algebra \cite{MAII} contains 
 a second proof,  
in Chapter 13 on the theory
of polynomial ideals, based on  van der Waerden's earlier paper \cite{vdW26}.
The proof  uses specialisation in fields of algebraic functions and
avoids elimination theory.
In the paperback  edition Algebra II \cite{vdW67} of the Moderne Algebra
the chapter on elimination
theory is eliminated; only the resultant of two polynomials in one variable has been
retained and moved to the first volume \cite{vdW71Alg}.

We witness here
the rise of commutative algebra as separate discipline, in a period which starts
with Noether and ends with Noether, to borrow
from the title of the essay \cite{Gray97}.
An important motivation for Lasker and Macaulay was the generalisation of
Max Noether's ``fundamental theorem on algebraic
functions''.  The first proved special case of the Nullstellensatz (Netto's theorem
\cite{Netto}) follows from Bertini's refinement of Noether's theorem  \cite{BerN}.
The most most far-reaching generalisation is Hentzelt's Nullstellensatz, proved by
Emmy Noether's  first official PhD student
Grete Hermann \cite{Hermann}. Only after this theorem Hilbert's theorem is referred
to as Nullstellensatz.
The influence of Emmy Noether on van der Waerden is well known. She also  
influenced Rabinowitsch: Noether 
spent the winter 1928/29 in Moscow and led a seminar on algebraic geometry
at the Communist Academy 
\cite{Ale36}.

From the early proofs we first give  Rabinowitsch's proof, which does not
prove the weak Nullstellensatz. We describe Kronecker's elimination theory,
from which the result does follow. It also gives ingredients for Hilbert's original
proof of the theorem. These proofs use induction, but otherwise
not more than the basic properties of the resultant of two binary forms. 
We recall those in the Appendix.
Less elementary are the proofs using primary decomposition.
We describe the background needed. The last proof is van der Waerden's proof
by specialisation. Finally we formulate but do not prove Hentzelt'z Nullstellensatz.
Before discussing the proofs we place them in a historic context and describe the
main characters and 
main developments, from Noether's fundamental theorem until Noether's
modern algebra.

\section{Notation}
We use modern terminology, in particular the word ideal for
the older term module, or modular system. Lasker \cite{L} uses the term module
for ideals in $\C[x_1,\dots,x_n]$, whereas his ideals are ideals in $\Z[x_1,\dots,x_n]$.
The use of ideal was propagated by Emmy Noether, see \cite[Footnote 11]{HN}. 
The name modular system is explained by the notation  introduced by
Kronecker \cite[\S 20]{Kronecker}. He writes
\[
G \equiv 0 \bmod (F_1,F_2,\dots,F_k)
\]
to express that a polynomial $G$ can be written in the form
\[
P_1F_1 + P_2F_2 + \dots + P_kF_k\;,
\]
that is,  $G\in  (F_1,F_2,\dots,F_k)$.
This is an extension of the notation $a\equiv b \bmod c$ introduced by Gauss for integers
to express that $a-b$ is divisible by $c$
without remainder. 

We  do not use the arithmetic 
terminology that an ideal $a$ divides $b$ if $b \subset a$. Macaulay \cite{M16} 
says in this case that $b$ contains $a$.

The term polynomial is used for the the older term  whole rational function. Homogeneous
polynomials will also be called forms. The coefficients will be taken in $K$, an 
algebraically closed field. In the older literature the base field is tacitly assumed to be
that of the complex numbers. Most proofs work in the general case.

We have tried to use a uniform notation. 
In particular, following Emmy Noether we use a Fraktur font to denote
ideals. Hermann Weyl addressed Noether at her funeral service:
``Mit vielen
kleinen deutschen Buchstaben hast Du Deinen Namen in die Geschichte der
Mathematik geschrieben''\footnote{With many German lower-case letters you wrote your name in the history of mathematics.} (this is the version   
Grete Hermann gives in a letter  to van der Waerden in 1982 \cite[p. 20]{Roquette};
the originaltext of the speech is also in loc. cit.).

Throughout this paper we use the name Nullstellensatz, although it is
of relatively recent date, see the MathOverflow question \cite{391718}.
Before \cite{vdW26} the theorem was referred to 
 as a known theorem of Hilbert, or  an explicit  reference was given to \cite[\S 3]{Hilb93}.
Macaulay \cite{M16} calls it the  Hilbert-Netto Theorem.
Interestingly van der Waerden refers to it as  a well known theorem of Hilbert in the paper
 \cite{vdW31} (an extract of a letter to  J. F. Ritt), which has the footnote
 ``See Macaulay, Modular Systems, p. 46. (J. F. R.)''. The first mention of a Nullstellensatz
concerns that of Hentzelt, in   Noether's
report on the thesis  of Grete Hermann (Noether, 1.2.1925, Promotionsakte Hermann)
\cite[p. 320]{Koreuber}.
 
The use of the German name originates in the US. Emmy Noether advised her students
``don't bother to translate, just read the German'', according to  Ruth McKee
\cite{Kee}.
Hodge and Pedoe  \cite{HP} use the term Hilbert's zero-theorem, but 
Zariski states in his review in Math Reviews [MR0028055] that they prove Hilbert's Nullstellensatz.
 Miles Reid   explains the name  as theorem of zeros, adding:
``but stick to the German if you don't want to be considered as  ignorant peasant''
 \cite{Miles}.

\section{From Noether to Noether}

\subsection{Noether's fundamental theorem}
Lasker spends nine pages of his paper \cite{L} to give an overview of the development of ideal theory. It was Max Noether who with his  fundamental theorem ``clearly and sharply
apprehended and demonstrated the central position, 
which the method of modular systems has in all questions of algebra'' \cite[p. 44]{L}.

Max Noether's fundamental theorem on
algebraic functions  \cite{N73} concerns plane curves.
The equation of an algebraic curve $f=0$ passing through the
intersection of two curves $\varphi=0$, $\psi=0$ can be written in the form
$f=A\varphi+B\psi$, if the intersections are simple, but as Noether noticed, the statement 
ceases to be valid otherwise.
An easy example is provided by two conics, tangent in two points.
Then the line passing through these two points 
is not contained in the pencil
spanned by the two conics, but twice the line (given by the square of a linear function) is.
The correct condition for $f$ to lie in the ideal $(\varphi,\psi)$ is that in each of the
intersection points  $f$ can be written as $f=a\varphi+b\psi$ with $a$ and $b$ power series;
it suffices to check this identity up to a certain order $\rho$ only depending
on the ideal $(\varphi,\psi)$, as was first observed by Bertini \cite{Ber}.

In a paper  ``Zur Theorie der Elimination'' \cite{Netto}
Netto gives a different answer
for the case of non-simple intersection, without making a connection 
to Noether's theorem; this is done  by
F. Meyer in his Jahrbuch review [JFM 17.0096.01].
Netto expresses his result geometricallyas follows: if an algebraic curve 
$f(x, y) = 0$ passes through all intersection points of two other algebraic curves 
$\varphi(x, y) = 0$, $\psi(x, y) =0$, then some power of the polynomial $ f(x, y)$ can
be expressed as  linear homogeneous function of $\varphi(x, y)$ and $\psi(x, y)$, i.e.
\[
f(x,y) ^\rho =A(x,y)\varphi(x,y)+B(x,y)\psi(x,y)\;,
\]
where $A(x,y)$ and $B(x,y)$ are also polynomials. As Hilbert \cite{Hilb93} remarks, this is the special case of his Nullstellensatz for two inhomogeneous variables. Netto's proof gives that the
required power is bounded by the highest intersection multiplicity.

To generalise Noether's fundamental theorem  to $n$ dimensions was one of the problems 
van der Waerden worried about when  he came to Göttingen in 1924.
In his paper \cite{vdW75} on the sources of his Moderne Algebra he says 
that a new world opened for him. Already in Amsterdam van der Waerden
``discovered that the real difficulties of algebraic geometry cannot
be overcome by calculating invariants and covariants'' \cite[p. 32]{vdW75}.
In Göttingen he learned from Emmy Noether that suitable tools 
had  been developed by Dedekind
and Weber, by Hilbert, Lasker and Macaulay, by Steinitz and by
Emmy Noether herself. 

Noether's theorem was generalised by König, Lasker and Macaulay.
The most general form is Hentzelt's Nullstellensatz, which provides a criterion indicating
how much a polynomial has to vanish in the zeros of an ideal in order  to
belong to it.  It was proved by Emmy Noether's  first PhD student
Grete Hermann. 
Van der Waerden gives a non-constructive proof in \cite{vdW26}.

\subsection{Kronecker's elimination theory}
The word ideal has its origin in Kummer's work on the
factorisation of algebraic numbers. His theory has been developed 
by Kronecker,  and in a different way by Dedekind. The story is told in detail
by Edwards \cite{Ed80}. A more elementary account is in Klein \cite[Ch. 7]{Klein}.
Kronecker's role in development of the algebraic tools in the period from
Noether to Noether is discussed in  \cite{Gray97}.

Kronecker's theory of algebraic quantities applies not only to number fields
but also to polynomial rings, and in particular to  Elimination Theory.
It gives a method to find all solutions of a set of algebraic
equations by successive eliminations. We describe it in Section \ref{kron-elim-thy}.

Kronecker lectured on his theory and disseminated it in private conversations, but 
 published his results   only in 1882 in a Festschrift
for Kummer's doctor jubilee  
\cite{Kronecker}.
A part of Kronecker's theory was treated in detail in a long paper in Acta Mathematica
by his student Jules Molk \cite{Molk}.
In this paper he mentions the result of Netto, which prompted Netto to publish
his theorem also in Acta Mathematica \cite{Netto}. Netto gives in the second 
volume of his book \cite{Netto-Alg} an extensive account of elimination theory, also of the older work
of Bézout, Liouville, Poisson and Cayley. 
He generalises his theorem from \cite{Netto} to the case of a complete intersection
of $n$ variables, adding that for $n=2$ Noether had shown this in a bit different,
for his purposes less suitable form. Indeed, by Bertini's bound on Noether's theorem
it follows that a sufficiently high  power of the function satisfies the condition
\cite{BerN}. With a flawed argument Netto goes  on to show that the result still holds
if there are more equations than variables, and derives the Nullstellensatz
with Hilbert's argument for zero sets of higher dimensions (see Section \ref {Hilb-proof}).

Kronecker's theory (and much more) was presented in the 564 page book \cite{Koe03} by 
Julius (Gyula) König, published simultaneously in Hungarian \cite{Koe-hun} and German. 
His goal (as stated in the Introduction of \cite{Koe03}) was to  popularise Kronecker's ideas
(``if this expression for this difficult  area of mathematics
is allowed''). He indeed succeeded, as shown by the many comtemporary
references. A footnote to the published paper \cite{M05} of Macaulay's talk at the ICM in Heidelberg (1904)
states that Professor Noether and Professor
Brill have kindly drawn his attention to the recently published book by König. This work
is ``remarkable for its precision and comprehensiveness and the large additions
it makes to the subject''. 
König's book is still mentioned in a footnote in \cite[p. 167]{vdW67}.
Nowadays it is almost completely forgotten. As \cite{Gray97} puts it,
a useful book on a new field will, if succesful, draw others to discover better results,
simpler and more general methods and if it does not become a classic the work
will gradually be covered up and forgotten. For a modern reader König's book is hard to read.

König gives applications to geometry, notably a generalisation
of Noether's theorem to higher dimensions.  He also treats several fundamental
results of Hilbert. In particular, he gives  
a new, simpler proof of the Nullstellensatz; but his proof is flawed, as he ``makes
an absurdly false assumption concerning divisibility'' \cite[p. 35]{M16}. Actually,
König gives a faulty proof  of an absurd divisibility statement \cite[p. 399]{Koe03}.

\subsection{Hilbert}
The next important development comes from Hilbert's work on invariant
theory, in his two famous papers in Mathematische Annalen 
\cite{Hilb90, Hilb93}.  Klein \cite[p. 329]{Klein} writes that Hilbert takes up ideas from Kronecker with Dedekind's way
of thinking, and applies them brilliantly on the problems of invariant theory.
Indeed, Hilbert states explicitly in the Introduction of  \cite{Hilb93} that he uses 
methods of the general theory of polynomial ideals, so that the theory of invariants
becomes a striking example  of that theory, just as  cyclotomic
fields constitute a striking example in number theory, where the most important
theorems about general number fields have first been found and proven. 
In the first paper \cite{Hilb90}, where he proves the basis theorem
(stating that every ideal is finitely generated) and his syzygy theorem and introduces
the Hilbert polynomial,
many geometric examples are given, and only in the last section the results
are applied to prove the finiteness of the system of invariants. One of the examples
solves a problem of Salmon.
Lasker  points to the influence of the work of Salmon and Cayley and comments:
\begin{quote}
Man hat Salmons Werk unterschätzt, weil seinen Methoden die Strenge der 
Beweisführung abging. Wie gro\ss\ dieser Fehler auch sein mag, so darf man niemals
die Bedeutung Salmons als des gro\ss en Problemstellers und Wegweisers vergessen.%
\footnote{Salmon's work has been underestimated because his methods lacked  rigor of
proof. However great this error may be, one must never
forget the importance of Salmon as the great problem poser and guide.}
(\cite[p. 44]{L})
\end{quote}
At several places Hilbert \cite{Hilb90} stresses the importance of generalising Noether's
fundamental theorem to higher dimensions. 

Hilbert formulates the basis theorem in a different way from what is usual nowadays. 
\begin{thm}
Given a non-terminating sequence of forms in the $n$ variables
$x_1,\dots,x_n$, say $F_1,F_2,F_3,\dots$,  there always exists a number 
$m$ such that every form of the sequence can be written as
\[
F=A_1F_1+A_2F_2+\dots+A_mF_m\;
\]
where $A_1,A_2,\dots,A_m$ are suitable forms in the same $n$ variables.
\end{thm}
Hilbert also  gives a second version of the basis theorem, 
for forms with integral coeffcients.
Hilbert's formulation seems natural if one thinks about the explicit computation of invariants
in special cases,
which leads to   lists. Moreover, 
Hilbert treats only homogeneous polynomials, or forms, whereas the modern formulation
works with inhomogeneous polynomials. The theorem can be extended to the inhomogeneous
case by making all polynomials homogeneous with a new variable of homogeneity
\cite[p. 38]{M16}.

Hilbert explicitly states  that the basis theorem applies in particular to
homogeneous ideals in polynomial rings; he uses Dedekind's term module. 
Hilbert makes the connection with Kronecker's concept of modular systems,
but stresses that his syzygy theory and the characteristic function (in modern
terms the Hilbert polynomial) use homogeneity in an essential way.

Conversely  the basis theorem for ideals, that every ideal in the polynomial ring $K[x_1,\dots,x_n]$ is finitely generated,  implies the   theorem in Hilbert's formulation
\cite[\S\ 80]{MAII}, \cite[\S\ 115]{vdW67}: the ideal generated by the $F_i$ has
a finite set of generators, each of which is a linear combination of only finitely many
$F_j$.

The basis theorem is the first step in proving that the ring of invariants is 
finitely generated. The invariants in question concern in modern terms
the action of the group $G=SL_n(\mathbb C)$ on a vector space $V$ which is a direct sum
of the type $S^{d_1}{\mathbb C^n}
\oplus \dots\oplus S^{d_k}{\mathbb C^n}$. The result is that  the ring
$\mathbb C[V]^G$ of $G$-invariant polynomials on $V$ is finitely generated.
By the basis theorem every invariant $i$ can be expressed as
$i=A_1i_i+\dots A_mi_m$. By an averaging procedure the $A_i$ can themselves
be taken as invariants, of lower degree than $i$.  By applying the same reasoning
to these invariants one finally obtains that $i$ is a sum of products of the $i_j$.
Nowadays one uses the Reynolds operator, which is a $G$-invariant projection
$\mathbb C[V] \to \mathbb C[V]^G$, but Hilbert had to construct it
using Cayley's $\Omega$-process; for details we refer to \cite{Derksen}.

Hilbert's proof was criticised for its nonconstructive
character. The goal of \cite{Hilb93} is to give a method to find the generators
(in principle). 
By the
homogeneous  version of Noether normalisation (proved by Hilbert for this purpose)
the ring of invariants is an integral extension of a polynomial ring $k[J_1,\dots,J_\kappa]$
with the $J_i$ invariants (of the same degree). The quotient field of the ring of invariants
is the field of rational invariants and by the theorem of the primitive element
it is an extension of the field $L=K(J_1,\dots,J_\kappa)$ of the form $L(J)$ with 
$J$ a polynomial invariant; Hilbert shows how to construct such a $J$. 
To find the ring of invariants  Hilbert gives three steps, of which the first is the 
most difficult, namely to find the system $\{J_1,\dots,J_\kappa\}$ of invariants, such that 
every  other invariant is integral over the $J_1,\dots,J_\kappa$, that is, satisfies a monic equation with coefficients which are polynomials 
in the $J_1,\dots,J_\kappa$. The second step is to find $J$, such that all invariants
are rational functions of $J,J_1,\dots,J_\kappa$. The third step is to find the 
integral elements of the field $L(J)$, which can be done according to a general theory
of Kronecker: ``If the  invariants $J, J_1, \dots, J_\kappa$ are known,  finding the full invariant system only requires the solution of an elementary problem from the arithmetic theory of algebraic functions'' \cite[p. 320]{Hilb93}.

The system $\{J_1,\dots,J_\kappa\}$ has the property that all invariants vanish if the
$J_1,\dots,J_\kappa$ vanish. Of fundamental importance for the whole theory is that
the converse can be proved: if invariants $I_1,\dots,I_\mu$ have the property that their
vanishing implies the vanishing of all other invariants, then every invariant is integral
over the $I_1,\dots,I_\mu$. To prove this Hilbert first shows the Nullstellensatz
(see Theorem \ref{Hilb-NSS} for Hilbert's version). This gives that  powers of the
generators of the ring of invariants lie in the ideal $(I_1,\dots,I_\mu)$ and therefore
every invariant of  degree at least some fixed $\rho$. The coefficients can again be taken as 
invariants. A finite number invariants of lower degree form therefore a basis of the ring
of invariants as $K[I_1,\dots,I_\mu]$-module, so this ring is integral over 
$K[I_1,\dots,I_\mu]$. Hilbert shows this with the now standard determinant trick.

A form for which all invariants vanishes is called a null-form. In the space of all forms they 
form an algebraic subset, and knowing it helps determining the invariants  $(I_1,\dots,I_\mu)$. For binary forms Hilbert determines the null-forms with elementary means:
a form $f(x,y)$ of degree $d$ is a null-form  if and only if $f$ has a zero of multiplicity
bigger than $\frac d2$.  This can easily shown with the Hilbert-Mumford criterion,
see \cite[Example 2.5.4]{Derksen}; Hilbert proved the criterion later in his paper to handle
forms of more variables. In fact,
this part of the theory was only taken up 70 years later by Mumford in his
Geometric Invariant Theory \cite{Mum}. For these developments and 
their relation to Hilbert's text we refer to the comments
by V.L. Popov to his Russian translation of Hilbert's paper \cite{popov}.

\subsection{Lasker}

Little is known about the origins of the   highly original 
paper ``Zur Theorie der Moduln and Ideale'' \cite{L}, by the world chess champion 
Emanuel Lasker.
Van der Waerden \cite{vdW75} states that Lasker took his Ph.D. degree
under Hilbert's guidance in 1905, but that is not correct.

Lasker (1868--1941) studied mathematics from 1888 in Berlin and later in Göttingen
(when Hilbert still was in Königsberg),
but 1891 he interrupted his studies and concentrated on chess, becoming world
champion in 1894. He took up his studies again in 1897, first in Heidelberg (taking 
courses with Landsberg)
and later in Berlin (courses with Hensel) \cite{rosen}.

Lasker submitted a manuscript for the Grand prix des sciences mathématiques
in 1898, where the question was 
``Chercher à étendre le rôle que peuvent jouer en analyse les séries divergentes'',
but it was considered to be a bit beside the question \cite{prix}.
He used the first 23 pages of this manuscript to get a doctoral degree. Max Noether  in Erlangen was prepared to help him.  
Staying at the Hotel  Erlanger Hof Lasker wrote to the dean on Monday January 29, 1900,
who convened the examining committee for the next Wednesday. On the same Monday
Noether already delivered his report. Lasker passed magma cum laude
\cite{Ullrich}.
Lasker submitted the paper 
to the Philosophical Transactions of the Royal Society of London, where it was
published in 1901 in German (Über Reihen auf der Convergenzgrenze) \cite{L01}.
So Lasker was neither a student of Hilbert nor of Noether.

Lasker wrote a small paper \cite{L04} on the theory of canonical forms, dated
New York  May 1903.
His main mathematical work \cite{L} might have been an attempt to start
an academic career;  he never had a permanent 
academic position in mathematics. The paper is dated Charlottenburg
March 1904. Right after Lasker travelled to the US to play the chess tournament
at Cambridge Springs.

Albert Einstein came to know Lasker in later life. He wrote on occasion of Lasker's sixtieth
birtday:
\begin{quote}
``Emanuel Lasker ist einer der stärksten Geister, denen ich auf meinem Lebenswege begegnet bin. Renaissance-Mensch, mit einem unbändigen Freiheitsdrang begabt, jeder sozialen Bindung abhold. So wurde er Schachmeister, wohl weniger aus besonderer hingebender Liebe zum Spiel. Letztere galt vielmehr der Philosophie, dem Verstehen überhaupt. Er liebt als ächter Eigenbrödler und Eigenwilliger die Deduktion und steht der induktiven Forschung fremder gegenüber. Kein Wunder es liegt ihm nicht, im Objekt den Richter über die Kinder seines Geistes zu sehen, sondern die Schönheit des Gedankens geht ihm über jene Wahrheit, die ihren Anspruch aus der Beobachtung des Objektes ableitet. Der Amor dei intellektualis 
ist sein einziger Gott, verkörpert in Mathematik und spekulativer Philosophie. Ich liebe seine Schriften unabhängig von ihrem Wahrheitsgehalt als die Früchte eines grossen originalen und freien Geistes.''\footnote{
Emanuel Lasker is one of the strongest minds I have encountered in the course of my life. A Renaissance man, gifted with a boundless desire for freedom, averse to any social obligation. Thus he became a chess master probably not so much because of any particular devoted love for the game. What he loves, rather, is philosophy, understanding in general. As a true maverick with a mind of his own, he loves deduction, and inductive research is foreign to him. That is not surprising he does not see the object as the judge of his mind’s offspring instead, for him the beauty of the idea is more important than the truth, which derives its claim from the observation of the object. The amor dei intellectualis is his sole god, embodied in mathematics and speculative philosophy. I love his writings independently of their truth content, as the product of a great original and free mind.} (\cite{einstein})
\end{quote}

Lasker's paper \cite{L} 
is famous for the introduction of primary ideals, but contains much more.
It  does not have an Introduction, but from the ``final remarks
 about some applications of the Theorems'' we can conclude that the 
 main objective of the theory is the extension of Noether's fundamental theorem
 to the case of several variables.
The paper contains a new approach to the Hilbert polynomial, 
based on multiplication with non-zero divisors; the Hilbert polynomial
is used for a new proof of the Nullstellensatz. 
There is also an extension
of the theory to the ring of convergent power series in several variables,
which  is used to prove Lasker's generalisation of Noether's theorem.
The last application sketched in the paper concerns Plücker formulas for curves
with arbitrary singularities.

\subsection{Macaulay}
Lasker proved that an ideal is the intersection of primary ideals, but gave no methods
to compute these. This was the goal of Macaulay's  paper \cite{M13}. 
F. S. Macaulay (1862--1937) was a school teacher 
until his retirement in 1911. For his mathematical work see \cite{EG23}.
  
While in \cite{M13} Macaulay uses 
the theories of Kronecker, Hilbert and Lasker, the goal of first three chapters of his 1916 Tract \cite{M16} is to 
present them.
In the preface Macaulay writes:
\begin{quote}
The present state of our knowledge of the properties of Modular Systems is chiefly due to the fundamental theorems and processes of L. Kronecker, M. Noether, D. Hilbert, and E. Lasker, and above all to J. König's profound exposition and numerous extensions of Kronecker's theory. (\cite[Preface]{M16})
\end{quote}

In this slim volume 
Macaulay only
treats the case of polynomial ideals in $\C[x_1,\dots,x_n]$, what he calls the algebraic
theory of modular systems; the ``absolute theory'' concerns the case of integer coefficients.
This is the same distinction as Lasker makes between modules and ideals.
The last chapter of the Tract introduces
the Inverse system, according to Paul Roberts (in his Introduction to the 1994 reprint)
one of the most original ideas in the book. A simplified treatment is given in
Macaulay's last paper \cite{M34} (not mentioned by Roberts). 
In her Zentralblatt review [Zbl. 0008.29105] Emmy Noether 
writes, among other things, that again the many examples and counter-examples
are important.

Macaulay's Tract was one of the works Emmy Noether advised B. L. van der Waerden
to study when he came to Göttingen in 1924.  It is the direct source for several sections in
Moderne Algebra, according to \cite{vdW75}. Elsewhere van der Waerden recollects: 
\begin{quote}
Most important
work on the theory of Polynomial Ideals was done by \textsc{Lasker}, the famous
chess champion, who had got his problem from \textsc{Hilbert}, and by 
\textsc{Macaulay},
a schoolmaster who lived near Cambridge, England, but who was nearly unknown
to the Cambridge mathematicians when I visited Cambridge in 1933. I guess
the importance of \textsc{Macaulay}'s work was known only in Göttingen. (\cite{vdW71})
\end{quote}

\subsection{Noether}
Meanwhile a different treatment of explicit elimination theory was given 
by Kurt Hentzelt in his 1914 Ph.D. thesis 
``Zur Theorie der Polynomideale und Resultanten'' under E. Fischer.
Kurt Hentzelt (1889--1914, he went missing in 
action near Diksmuide) studied first in Berlin and then 1919-1913 in 
 Erlangen \cite[p. 320]{Koreuber}.
Presumably he was also a student of Emmy Noether \cite[p. 15]{Koreuber}.
She acknowledges him in her paper on ``fields and systems of rational
functions'' (dated May 1914, before the start of the First World War) \cite{N15}.
In \cite{N21} she gives in a footnote the simplest example of a
non-unique  decomposition in primary ideals, adding that she got it from
K. Hentzelt. Noether published a
conceptual version \cite{HN} of Hentzelt's thesis, which she characterises:
\begin{quote}
Diese ganz auf Grund eigener Ideen verfaßte Dissertation ist lückenlos aufgebaut; aber
Hilfssatz reiht sich an Hilfssatz, alle Begriffe sind durch Formeln mit vier und fünf Indizes
umschrieben, der Text fehlt fast vollständig, so daß dem Verständnis die größten Schwierigkeiten
be\-reitet werden.'%
\footnote{This dissertation, entirely based on own ideas, is structured without gaps; but
Lemma follows lemma, all concepts are represented by formulas with four and five subscripts, the text is almost completely absent, so that the greatest difficulties are caused for
understanding.} (\cite[p. 53]{N21})
\end{quote}
The part concerning computation in a finite number of steps was reserved for a
later publication. Noether gave this problem to her first PhD student
Grete Hermann. 

In \cite{vdW26} van der Waerden gives a new foundation for the theory 
of zeros of polynomial ideals, independent of elimination theory.
Even though he later pleaded for
the use of elimination theory in algebraic geometry \cite{vdW54}, van der Waerden
contributed to the elimination of elimination theory.
In \cite{vdW26} he  also gives a new proof of  Hilbert's Nullstellensatz. 
Van der Waerden\cite{vdW71} recalls his use of generic points:
\begin{quote}
I wrote a paper \cite{vdW26} based upon this simple idea and showed it {\sc Emmy Noether}.
She at once accepted it for the Mathematische Annalen, without telling me that
she had presented the same idea in a course of lectures just before I came to
Göttingen. I heard it later from {\sc Grell}, who had attended her course. (\cite{vdW71})
\end{quote}

The above quote shows a greater participation of Emmy Noether  in the
development of algebraic geometry than visible from her published papers. 
Emmy Noether spent the winter 1928/29 in Moscow. She gave a course on abstract
algebra at Moscow University and led a seminar on algebraic geometry
at the Communist Academy (in 1936 merged with the Academy of Sciences of the Soviet Union)
\cite{Ale36}.
It is probable that J. L. Rabinowitsch was one of the participants. 
In her report on the thesis of Hans Fitting \cite[p. 316]{Koreuber} Noether mentions unpublished
work of ``Rabinowitsch-Moskau'' on the subject of \cite{N29}, which probably
her Moscow lectures took up.
Not much is known
about him, but it can be
Juli Lasarewitsch Rabinowitsch (in German transliteration), who was born in 1904
 and graduated from Moscow State University in 1924, where he worked since 1932 and was awarded a Ph.D. and a title of Docent in 1935 \cite{416577}.
 He later mainly worked on Differental Equations, but one can imagine Noether
attracting a rather wide audience.
Curiously van der Waerden always writes A. Rabinowitsch, in his Moderne Algebra
but also in \cite{vdW75}.
Zariski repeats this mistake in \cite{Zar}, which shows that his source is 
van der Waerden's book.

It is not unlikely that Noether took Rabinowitsch' paper back to Germany and arranged for
its publication in Mathematische Annalen, and that she provided the references.

\section{Rabinowitsch' paper}
The text of \cite{Rab29} consists of only 13 lines.
Figure \ref{rab-paper} shows it in translation.

\begin{figure}
\noindent\fbox{
\begin{minipage}{0.97\textwidth}
\bigskip
\renewcommand{\thempfootnote}{\arabic{mpfootnote}}
\centering
\begin{minipage}{0.9\textwidth}
\setlength{\parindent}{1.5em}

\centerline{\bf On Hilbert's Nullstellensatz.} 
\centerline{By}
\centerline{J. L. Rabinowitsch  in Moscow.}

\bigskip
\indent Theorem:
\textit{If the polynomial $f(x_1,x_2,\dots,x_n)$ vanishes in all zeros
--- in an algebraically closed field --- of a polynomial ideal
$\mathfrak{a}$, then there is a power $f^{\rho}$ of $f$ 
belonging to $\mathfrak{a}$.}
\medskip

Proof:
Let $\mathfrak{a}=(f_1,\dots,f_m) $, where $f_i$ contain the variables
$x_1,\dots,x_n$. Let $x_0$  be an auxiliary variable. 
We form the ideal $\bar{\mathfrak{a}}=(f_1,\dots,f_m,
x_0f-1)$. As  by assumption $f=0$ whenever all $f_i$ vanish, the ideal
$\bar{\mathfrak{a}}$ has no zeros.

Therefore $\bar{\mathfrak{a}}$ has to coincide with
 the unit ideal.
 (Cf. for example K. Hentzelt, ``Eigentliche Eliminationstheorie'',
 \S\ 6, Math. Annalen \textbf{88}\footnote{Follows also already from Kronecker's
 elimination theory.}.)
 If then $1=\sum_{i=1}^m F_i(x_0,x_1,\dots,x_n)f_i+F_0(x_0f-1)$
 and  we put 
 $x_0=\frac 1f$ in this identity, so results:
\[
1= \sum_{i=1}^m F_i\left(\frac 1f,x_1,\dots,x_n\right)f_i = 
\frac{\sum_{i=1}^m \bar{F}_if_i} {f^\rho}\;.
\]
Therefore $f^{\rho}\equiv 0\pmod{ \mathfrak a}$, q.e.d.
\bigskip
\end{minipage}
\bigskip

\centerline{\small(Received on 8. 5. 1929)}
\medskip
\end{minipage}}
\caption{Rabinowitsch's paper \cite{Rab29}}
\label{rab-paper}
\end{figure}

The statement  that a nontrivial ideal has zeros is nowadays
called weak Nullstellensatz, but in the text and earlier it is not considered to
be a special case. The proof of the result
can be found in many places (e.g. \cite[p. 21]{M16}), 
but  the connection with Hilbert's theorem
is not  made. 
Only in \cite{M13} Macaulay  gives a footnote to the statement
that the only ideal without zeros is the unit ideal, where he 
writes that this follows as a particular case of the theorem
in \cite[\S3]{Hilb93} but was known earlier from the theory of the resultant. 
In the accounts of Kronecker's elimination theory by Kronecker himself
\cite{Kronecker}, Molk \cite{Molk}, Netto \cite{Netto-Alg} and König \cite{Koe03}
the conclusion is not drawn.
In \cite{HN} we find 
in \S 6 on p. 76 the statement
 that Theorem XII of that paper  in particular shows that every ideal
without zeros becomes the unit ideal,  again without mentioning the
Nullstellensatz.

\section{Kronecker's elimination theory}\label{kron-elim-thy}
In this  section we explain Kronecker's theory, following
the account given by Macaulay \cite{M16}. We show in particular how it
implies the weak Nullstellensatz.

Let $F_1,\dots,F_k\in K[x_1,\dots,x_n]$ with $K$ an algebraically closed field
and consider the equations $F_1=\dots=F_k=0$.
 The problem is to find all solutions.
Coincidences due to special values of the coefficients, like equations not being 
regular (a polynomial of degree $l$ is regular in $x_1$ if the monomial $x_1^l$ occurs with 
non-zero coefficient), can be avoided by a (general)  linear change of coordinates.
As Macaulay remarks this transformation is seldom needed in specific examples,
but always assumed in theoretical reasoning. 

If  the
polynomials $F_i$ have a non-constant   greatest common divisor $D$, then the  
hypersurface $D=0$ gives solutions of the equations.
This greatest common divisor
can be found with the Euclidean algorithm, by considering
the polynomials as elements of $K(x_2,\dots,x_n)[x_1]$.
If $D$ is a constant we take it equal to $1$.
We divide the $F_i$ by $D$ and write $F_i=D\phi_i$.  We eliminate
 $x_1$ from the equations $\phi_1=\dots=\phi_k=0$, using the following device.
We form the expressions
 \begin{align*} 
\Phi_1 &=  u_1\phi_1+\dots +u_m\phi_m \\ 
\Phi_2 &=   v_1\phi_1+\dots +v_m\phi_m
\end{align*}
with the $u_i$ and $v_i$ indeterminates, and write the resultant $R(\Phi_1,\Phi_2)$
 (see Appendix \ref{appendix})
as 
\[
w_1 F_1^{(1)}+w_2 F_2^{(1)}+\dots +w_{k_1} F_{k_1}^{(1)}\;,
\]
where the $w_i$ are monomials in the $u_j$ and $v_j$ and $F_i^{(1)}\in 
K[x_2,\dots,x_n]$. Any solution of $F_1=\dots=F_k=0$ is a solution of $D=0$ or of
$\phi_1=\dots=\phi_k=0$, and any solution of $\phi_1=\dots=\phi_k=0$ is a solution of
$F_1^{(1)}=\dots= F_{k_1}^{(1)}=0$, since by the property \ref{res-in-id}
of the resultant 
 $\sum w_i F_i^{(1)}= 
A_1 \Phi_1 +A_2 \Phi_2$; by equating coefficients of the $w_i$ we conclude that
$F_i^{(1)}\in (\phi_1,\dots,\phi_k)$. Therefore $D\,F_i^{(1)}\in (F_1,\dots,F_k)$.
Conversely, if $(\xi_2,\dots,\xi_n)$ is a solution of $F_1^{(1)}=\dots= F_{k_1}^{(1)}=0$,
then the resultant $R(\Phi_1,\Phi_2)$ vanishes for 
$(x_2,\dots,x_n)=(\xi_2,\dots,\xi_n)$ and the equations $\Phi_1=\Phi_2=0$ have
a solution $(\xi_1,\xi_2,\dots,\xi_n)$ (and we find all solutions). As $(x_1-\xi_1)$ 
is a factor of $\Phi_1(x_1,\xi_2,\dots,\xi_n)$, it does not depend on the $v_i$; 
nor does it depend on the $u_i$, being a factor of $\Phi_2$.  Therefore 
$(\xi_1,\xi_2,\dots,\xi_n)$ is a solution of  $\phi_1=\dots=\phi_k=0$, so also of
$F_1=\dots=F_k=0$.

We may assume that the $F_i^{(1)}$ are regular in $x_2$; the needed linear
transformation could have been performed at the start.  We apply the same procedure
and find the greatest common divisor  $D^{(1)}$ of the $F_i^{(1)}=D^{(1)}\phi_i^{(1)}$
considered as polynomials in $x_2$, and eliminate $x_2$ to get  polynomials
$F_1^{(2)},\dots, F_{k_2}^{(2)}$ in $x_3,\dots,x_n$.
 Any solution of $F_1=\dots=F_k=0$ is a solution of $D\,D^{(1)}=0$ or of 
 $F_1^{(2)}=\dots= F_{k_2}^{(2)}=0$  and $D\,D^{(1)}F_i^{(2)}\in (F_1,\dots,F_k)$

We continue and successively find $D^{(j)}$ and eliminate $x_{j+1}$.
After eliminating $x_{n-1}$ we have polynomials $F_i^{(n-1)}$ in one variable
$x_n$, with greatest common divisor  $D^{(n-1)}$ and after dividing with this common factor
the polynomials $\phi_i^{(n-1)}$ have no common root.
We find that any solution of $F_1=\dots=F_k=0$
is a solution of the single equation $D D^{(1)}\dots D^{(n-1)}=0$. Conversely,
from the solutions of $D D^{(1)}\dots D^{(n-1)}=0$ we can find all solutions
of $F_1=\dots=F_k=0$. As the $D D^{(1)}\cdots D^{(n-2)}F_i^{(n-1)}
=D D^{(1)}\cdots D^{(n-1)}\phi_i^{(n-1)}$ lie in the ideal $(F_1,\dots,F_k)$
and $1\in (\phi_1^{(n-1)},\dots,\phi_k^{(n-1)})$,
we conclude that 
\[
D D^{(1)}\cdots D^{(n-1)}  \in  (F_1,\dots,F_k)\;.
\]

\begin{defn} $D D^{(1)}\cdots D^{(n-1)}$ is the complete (total) resolvent of the
equations $F_1=\dots=F_k=0$, and 
$D^{(i-1)}$ is the complete partial resolvent of rank $i$. Any factor of $D^{(i-1)}$
is a partial resolvent of rank $i$.
\end{defn}

The weak Nullstellensatz follows.
\begin{prop}[{\cite[p. 21]{M16}}]\label{WNSS}
If the equations $F_1=\dots=F_k=0$ have no solution
then the complete resolvent  is equal to $1$ and 
consequently $1\in (F_1,\dots,F_k)$. 
\end{prop}

Macaulay continues to show by examples that the resolvent does 
not always detect embedded components or may indicate such when they do not exist.
This problem does not occur with Hentzelt's elimination theory. Noether  \cite{HN}
describes how to form a resultant form with better properties, depending only on the
ideal.  We refer for details to Krull's report in the 1939 edition of the Enzyklopädie
\cite{Kr-enz}. 
Whereas Kronecker's method seeks to solve the equations
$F_1=\dots=F_k=0$, Hentzelt looks for zeros of the ideal 
$\mathfrak a=(F_1,\dots,F_k)$.  
We may suppose that $\mathfrak a$ contains a polynomial
$F$, which is regular in $x_1$ of order $r$. Let $R_1,\dots,R_t$ denote the remainders
obtained by dividing the polynomials $x_1^jF_i$, $j=0,\dots,r-1$, $i=1,\dots,k$,
by $F$, as polynomials in $x_1$. 
Let  $M$ be the set of all polynomials in $\mathfrak a$ with degree less than $r$ in $x_i$.
It is a submodule of the free $K[x_2,\dots,x_n]$-module with basis $1,\dots,x_1^{r-1}$,
with $R_1,\dots,R_t$  generators of $M$.
 The rank of $M$ is less than $r$ if and only if
the polynomials in $I$ have a common factor of positive degree in $x_1$; this holds for
undeterminate $x_2,\dots,x_n$ but also  for specialised values. 
Let $\mathfrak a_1$ be the ideal of the minors of size $r$ of the coefficient matrix  of $M$.
Then $\mathfrak a_1\subset \mathfrak a$  and 
$\mathfrak a_1$ depends only on $\mathfrak a$.
If $\mathfrak a_1$ is not the zero ideal, then we can proceed in the same way. 
This process stops
if some ideal $\mathfrak a_r=0$, or with $\mathfrak a_n$. If $\mathfrak a_r=0$, then $x_{r+1},\dots,x_n$ can be chosen
arbitrarily, and the value of the other variables can be found by successively solving
equations.
 As $\mathfrak a_n$ does not depend on the variables,
it can only be the zero or unit ideal. 
In particular, if $\mathfrak a$ has no zeroes, then it is the
unit ideal \cite[p. 76]{HN}.

\begin{ex}To illustrate the difference in the elimination procedures we
consider Macaulay's example  iii in section 17 \cite[p. 23]{M16}. He considers the ideal 
$\mathfrak a=(x_1^3,x_2^3,x_1^2+x_2^2+x_1x_2x_3)$.
A less symmetric,
but more convenient basis of $\mathfrak a$ is 
\[
(x_1^2+x_2^2+x_1x_2x_3,x_1x_2^2(1-x_3^2),x_2^3)\;.
\]
The ideal $\mathfrak a$ has one
isolated component, $\mathfrak a'=(x_1^2+x_2^2+x_1x_2x_3,x_1x_2^2,x_2^3)$
and two embedded components $\mathfrak a''=(x_3-1,x_1^2+x_2^2+x_1x_2,x_2^3)$
and $\mathfrak a'''=(x_3+1,x_1^2+x_2^2-x_1x_2,x_2^3)$.
As the polynomial $f_1=x_1^2+x_2^2+x_1x_2x_3$ is regular in $x_1$, we may in
Kronecker's method take the resultant of $f_1$ and $v_2f_2+v_3f_3$, where $f_2$
and $f_3$ are the other two generators of $\mathfrak a$  \cite[\S 73, Remark 1]{MAII}.
We get the determinant
\[
\begin{vmatrix}
1 & x_2x_3 & x_2^2\\
v_2 x_2^2 (1-x_3^2) & v_3x_2^3&0\\
0& v_2 x_2^2 (1-x_3^2) & v_3x_2^3
\end{vmatrix}
\]
which equals
\[
x_2^6\left(v_2^2(1-x_3^2)^2-v_2v_3x_3(1-x_3^2)+v_3^2\right)\;.
\]
It follows that the complete resolvent is $x_2^6$.
 For $\mathfrak a'$ the computation is almost the
same, except that the factors $(1-x_3^2)$ are to be removed from the determinant.
Although $\mathfrak a\varsubsetneq \mathfrak a'$ both ideals have the same complete resolvent.

To eliminate according to Hentzelt-Noether we divide $f_2,x_1f_2,f_3,xf_3$ by 
$f_1=x_1^2+x_2^2+x_1x_2x_3$ and find that $f_2,x_1f_2,f_3$ form a basis
of the module of polynomials in $\mathfrak a$ of degree at most 1 in $x_1$.
The coefficient matrix is
\[
\begin{bmatrix}
x_2^3&0&0\\
0&x_2^3&  x_2^2 (1-x_3^2)
\end{bmatrix}
\]
with minors $x_2^6$ and $x_2^5(1-x_3^2)$. For $\mathfrak a'$ we find the ideal 
generated by $x_2^5$.
\end{ex}

Successive elimination requires that the variables are general. 
This is achieved by
Noether \cite{N23} 
by adjoining the coefficients $u_{ij}$ of an indeterminate $n\times n$ matrix $U$ to the
field $K$ and the change of variables $y_i=\sum u_{ij}x_j$. The same device is 
used by Hermann \cite{Hermann}.

\begin{remark}
Proposition \ref{WNSS} is here a consequence of a more general theory.
A direct proof of the weak Nullstellensatz
is shorter \cite[\S 74]{MAII}. We use induction on the number of variables.
For one variable the result is true by the extended Euclidean algorithm.
Suppose the ideal $\mathfrak a=(F_1,\dots,F_k)$ has no common zeros, and that 
$\mathfrak a$ contains a polynomial 
$F$  regular in $x_1$. We eliminate $x_1$ as above (with Kronecker's or 
Hentzelt's method) and find polynomials
$F_1^{(1)},\dots, F_{k_1}^{(1)}\in  (F_1,\dots,F_k)$, again without common zeros.
By the induction hypothesis $1\in (F_1^{(1)},\dots, F_{k_1}^{(1)})
\subset  (F_1,\dots,F_k)$.
\end{remark}

\section{Hilbert's proof}\label{Hilb-proof}
Hilbert's original proof is also based on elimination. The theorem is formulated
with the application to invariants in mind. It looks different from the theorem 
stated by Rabinowitsch.

\begin{thm}\label{Hilb-NSS}
Given $m$ homogeneous polynomials $f_1,\dots,f_m$ in
$n$ variables $x_1,\dots,x_n$, let $F, F', F'', \dots$ be homogeneous 
polynomials in the same
variables with the property that they vanish for all those values of these
variables for which the given  $m$ polynomials $f_1,\dots,f_m$ all
are equal to zero: then it is always possible to determine an integer $r$
such that any product $\prod^{(r)}$ of $r$ arbitrary polynomials
of the sequence $F, F', F'', \dots$ can be expressed in the form
\[
\textstyle{\prod^{(r)}}=a_1f_1+a_2f_2+\dots+a_mf_m\;,
\]
where $a_1,a_2,\dots,a_m$ are suitably chosen polynomials in the
variables $x_1,\dots,x_n$.
\end{thm}

\begin{remark}
\noindent
\begin{enumerate}[label=\arabic*.,wide,nosep] 
\item Hilbert formulates the projective Nullstellensatz.
As  the polynomials are homogeneous, their zeros are taken in projective space.
The inhomogeneous version follows by making all polynomials in the sequence
homogeneous with a new variable  $x_0$, and applying the homogeneous theorem to
a set of generators of the homogenisation of the ideal $(f_1,\dots,f_m)$.
Putting then $x_0=1$ leads to the sought relation.
\item
The Theorem in the above form implies
in particular that for any polynomial among the $F, F', F'', \dots$
the $r$-th power lies in the ideal $(f_1,\dots,f_m)$. Hilbert remarks that this fact was 
stated and proved for inhomogeneous polynomials of two variables
by Netto \cite{Netto}. 
The special case that the $r$-th power lies in the ideal implies the general
case.
Firstly, by Hilbert's basis theorem (Theorem I of \cite{Hilb90}), the polynomials
in the sequence are expressable in finitely many of them, say
$F^{(1)}, \dots,F^{(k)}$. Any product of $r$ polynomials
$F, F', F'', \dots$ becomes a sum of products of $r$ polynomials
$F^{(1)}, \dots,F^{(k)}$ with polynomials coefficients.
If $(F^{(i)})^{r_i}\in (f_1,\dots,f_m)$, then put
$r=(r_1-1)+(r_2-1)+\dots+(r_k-1)+1$. Every product
$(F^{(1)})^{l_1} \cdots (F^{(k)})^{l_k}$ with $\sum l_i=r$ contains at least
one factor  $(F^{(i)})^{r_i}$; otherwise $\sum l_i\leq \sum(r_i-1) = r-1$.
\cite[\S 75]{MAII}, \cite[\S130]{vdW67}.
\item 
The statement
that $r$ is independent of $f$, but only depends on the $f_i$, is nowadays   normally
not included. It follows from the fact that the ideal of polynomials vanishing at the zero set
of the $f_i$ is finitely generated. 
\end{enumerate}
\end{remark}

For the proof of the theorem we first reduce to the case that there are only
finitely many $F^{(i)}$. By the basis theorem every polynomial in the sequence
is a linear combination  of say $F^{(1)}, \dots,F^{(k)}$, and a product of $r$
polynomials is a sum  of products of $r$ polynomials
$F^{(1)}, \dots,F^{(k)}$. Hilbert writes this reduction at the end of his proof,
and starts by assuming that there are only finitely many polynomials in the sequence.

The proof then splits in two parts. In the first it is assumed
that the polynomials
$f_1,\dots,f_m$ have only finitely many common zeros.
The main tool is elimination using the resultant of two binary forms (see Appendix \ref{appendix}).
Substituting in the $f_i$  the expressions
$x_1\xi_1,\dots,x_{n-1}\xi_1,\xi_2$ for the variables $x_1,\dots,x_n$
makes them binary forms
in the variables $\xi_1$, $\xi_2$, of degrees  $\nu_1,\dots,\nu_m$.
Let $\nu = \max_j\{\nu_j\}$.

To eliminate $\xi_1,\xi_2$ Hilbert uses Kronecker's device \cite{Kronecker} and forms the expressions
\begin{align*} 
F_1 &=  u_1f_1+\dots +u_mf_m \\ 
F_2 &=   v_1f_1+\dots +v_mf_m
\end{align*}
where the $u_i$ and $v_i$ are binary forms in  $\xi_1$, $\xi_2$ of degree $\nu-\nu_i$
with indetermined coefficients, making
$F_1$ and $F_2$ homogeneous of degree $\nu$.
The resultant $R(F_1,F_2)$ is a polynomial in the indeterminates occurring
in the $u_i$ and $v_i$, whose coefficients are polynomials
$f_1',\dots,f_{m'}'$ depending only on the variables $x_1,\dots,x_{n-1}$,
and by putting $\xi_1=1$, $\xi_2=x_n$ 
one sees (from Proposition \ref{polxres-in-id}) that the $f_i'$  lie in the ideal $(f_1,\dots,f_m)$.

Hilbert does not consider the possibility that all $f_i'$ are identically zero.
This happens if one of the common zeros is the point $(0:\dots:0:1)$,
for then $(\xi_1:\xi_2)=(0:1)$ is a common zero of $F_1$ and $F_2$.
The standard trick is to apply a general linear transformation.
We may therefore assume that no common zero of the transformed system
lies in a coordinate hyperplane, simplifying somewhat Hilbert's argument.

If the polynomials $f_i$ have the common zeros
$(\alpha_1:\ldots:\alpha_{n-1}:\alpha_n)$,
$(\beta_1,:\ldots:\beta_{n-1}:\beta_n)$, \dots,
$(\kappa_1:\ldots:\kappa_{n-1}:\kappa_n)$, then the $f_i'$ have 
only the common zeros
$(\alpha_1:\ldots:\alpha_{n-1})$, \dots,
$(\kappa_1:\ldots:\kappa_{n-1})$.

In the same way the variable $x_{n-1}$ can be eliminated from the
$f_i'$, leading to polynomials $f_1'',\dots,f_{m''}''$, and so on until
a system of binary forms $f_1^{(n-2)},\dots, f_{m^{(n-2)}}^{(n-2)}$
in the variables $x_1$, $x_2$ is reached.

Hilbert uses this procedure to prove the result by induction on the
number of common zeros.
The base of the induction is the case that the $f_i$ have no common zeros
at all. Then the  forms $f_1^{(n-2)},\dots, f_{m^{(n-2)}}^{(n-2)}$ have 
no common zeros. This implies that every binary form of sufficiently high degree
lies in the ideal $(f_1^{(n-2)},\dots, f_{m^{(n-2)}}^{(n-2)})$
and therefore in the ideal $(f_1,\dots,f_m)$; in particular $x_1^{r_1}$ and
$x_2^{r_2}$ lie in the ideal for some $r_1$ and $r_2$. In the same way it
follows that $x_3^{r_3}$, \dots, $x_n^{r_n}$ lie in the ideal for sufficiently
large $r_3$, \dots, $r_n$. Therefore every homogeneous polynomial
in $x_1, \dots, x_n$ of degree at least $\sum(r_i-1)+1$ lies in the ideal, proving
the base case. 

This result can be called the weak projective Nullstellensatz, and we formulate
it separately.
\begin{prop}\label{WPNSS}
 A  homogeneous ideal $\mathfrak a$ has no zeros  if and only if there exists an 
integer $r$ such that 
every form of degree at least $r$ lies in $\mathfrak a$.
\end{prop}

We continue with the proof of the theorem.
The induction step is that the statement holds if the polynomials
have a given number of common zeros, say  $(\beta_1:\ldots:\beta_{n-1}:\beta_n)$, \dots,
$(\kappa_1:\ldots:\kappa_{n-1}:\kappa_n)$. Suppose that there is an additional common zero $(\alpha_1:\ldots:\alpha_{n-1}:\alpha_n)$. Then every $F^{(i)}$ can be written
in the form{\small
\[
(\alpha_2x_1-\alpha_1x_2)F^{(i)}_{12}+
(\alpha_3x_1-\alpha_1x_3)F^{(i)}_{13} + \dots +
(\alpha_nx_{n-1}-\alpha_{n-1}x_n)F^{(i)}_{n-1,n}.
\]}

\noindent
By our assumption  $\alpha_1$ and $\alpha_2$ are both non-zero, so elimination as above leads to forms
 $f_1^{(n-2)},\dots, f_{m^{(n-2)}}^{(n-2)}$
in the variables $x_1$, $x_2$ lying in the ideal $(f_1,\dots,f_m)$,
which only have the common zeros
$(\alpha_1:\alpha_2)$, \dots, $(\kappa_1:\kappa_2)$. Choose one of these forms and write
it in the form $(\alpha_2x_1-\alpha_1x_2)^{r_{12}}\varphi_{12}$ with $\varphi_{12}$
a binary form not vanishing for $x_1=\alpha_1$, $x_2=\alpha_2$. 
In the same way one finds
$r_{ij}$ and $\varphi_{ij}$ for the other $1\leq i <j \leq n$.

Put $r'=r_{12}+r_{13}+\dots+r_{n-1,n}$  and $\Phi=
\varphi_{12}\,\varphi_{13}\cdots\varphi_{n-1,n}$. Then 
\[
\Phi\,\textstyle{\prod^{(r')}}\in  (f_1,\dots,f_m)\;,
\]
where $\Phi$ is a polynomial that does not vanish in the point 
$(\alpha_1:\ldots:\alpha_n)$ and $\prod^{(r')}$ is an $r'$-fold product of $F^{(i)}$.
The polynomials $\Phi,f_1,\dots,f_m$ have only the common zeros $(\beta_1:\ldots:\beta_n)$, \dots,
$(\kappa_1:\ldots:\kappa_n)$. Therefore there exists a number $r''$
such that  $\prod^{(r'')}\in (\Phi,f_1,\dots,f_m)$.
Then  $\prod^{(r)}\in (f_1,\dots,f_m)$ for  $r=r'+r''$, which proves the induction step.
 
In the second step the theorem is proved  in general by induction on the number
 of variables.
 The induction hypothesis is that the result holds for $n-1$ variables and that the
 number $r$  can be chosen below a bound only depending on
 the degrees and the (finite) number of the forms $f_1,\dots,f_m,F,F',\dots$, but not on their
 coefficients.
 The base of the induction is $n=2$, where the result is true by the first part of the proof, as binary forms can only have 
 a finite number of zeros.

For the induction step put $x_1=tx_2$. The polynomials 
$f_1,\dots,f_m,$ $F,F',\dots$ become polynomials $g_1,\dots,g_m,G,G',\dots$
in the $n-1$ variables $x_2,\dots,x_n$ with coefficients polynomials in the 
parameter $t$. If $t$ takes a specific value then every $G^{(i)}$ vanishes whenever
all $g_i$ vanish (as polynomials in  $x_2,\dots,x_n$). By the induction hypothesis
there is a number $r_{12}$ such that every product $\prod^{(r_{12})}$ 
of polynomials  $G^{(i)}$ for every special value of $t$ has a representation
\[
\prod^{(r_{12})}=b_1g_1+\dots +b_mg_m
\]
with the $b_i$ polynomials in $x_2,\dots,x_n$. 
Considering  the coefficients of the $b_i$ as indeterminates $u_j$ and taking in this
equation coefficients of the monomials in  $x_2,\dots,x_n$ yields a system
of linear inhomogeneous equations for the $u_j$. The coefficients of these
linear equations are polynomials in $t$, and for every value of $t$ solutions exist.

At this point Hilbert uses an ``easily proved'' lemma:
If  a given a system of linear equations
\begin{align*} 
c_{11} u_1+\dots +c_{1p}u_p &= c_1, \\ 
&\vdotswithin{=}\\
c_{q1} u_1+\dots +c_{qp}u_p &= c_q
\end{align*}
with $c_{ij}$, $c_k\in K[t]$, has solutions for every value of $t$ ($K$ being
an infinite field)
then there exists a solution with $u_i\in K(t)$.

Indeed, as Popov remarks in his comments \cite{popov}, a solution exists if and only if
 the rank of
the coefficient matrix is equal to the rank of the augmented matrix.
As $k$ is infinite, one can find a $t_0\in k$ such the rank of these matrices 
over $k(t)$ is the same as the rank over $k$ with $t_0$ substituted for $t$.
Applying this lemma to the equations for the coefficients of the $b_i$
and substituting $t=\frac{x_1}{x_2}$ gives after clearing
denominators that 
\[
\psi_{12}\prod^{(r_{12})} \in (f_1,\dots,f_m)
\]
with $\psi_{12}$ a binary form in $(x_1,x_2)$, and 
$\prod^{(r_{12})}$ the product of $r_{12}$ polynomials $F^{(i)}$
corresponding to the chosen $G^{(i)}$.
In the same way one finds $r_{ij}$ and a binary form 
 $\psi_{ij}$ in $(x_i,x_j)$ with 
 $\psi_{ij}\prod^{(r_{ij})} \in (f_1,\dots,f_m)$.
Now put $r'=\max\{r_{ij}\}$ and choose $r'$ polynomials $F^{(i)}$.
The corresponding binary forms $\psi_{12},\dots,\psi_{n-1,n}$
have only finitely many common zeros, so also the
polynomials $\psi_{12}\dots\psi_{n-1,n},f_1,\dots,f_m$.
By the first part of the proof there exists a number $r''$ such that 
\[
\prod^{(r'')} \in (\psi_{12}\dots\psi_{n-1,n},f_1,\dots,f_m)\;.
\]
As $\psi_{ij}\prod^{(r_{ij})} \in (f_1,\dots,f_m)$ one has that for
$r=r'+r''$ that  
\[
\prod^{(r)} \in (f_1,\dots,f_m)\;.
\]
This concludes the induction step, and with that the proof of the Nullstellensatz.

\section{Proofs using primary decomposition}
Primary ideals were introduced by Lasker \cite{L}
in the setting of  polynomial rings. He used primary decomposition
to give a new proof of the Nullstellensatz. 
Macaulay \cite{M16} follows this strategy, but with different proofs.

The modern definition of primary ideals is due to Emmy Noether \cite{N21},
and applies to all Noetherian rings. 

\begin{defn} 
An ideal $\mathfrak q$ in a Noetherian ring $R$ is primary if  whenever $ab\in \mathfrak q$ but 
$a\notin \mathfrak q$ it follows that  $b^k\in \mathfrak q$ for some $k>0$. 
\end{defn}

The radical $\sqrt \mathfrak q$ of $\mathfrak q$, that is $\{a\in R\mid a^k \in \mathfrak q \text{ for some }k>0\}$,  is 
a prime ideal $\mathfrak p$ and $\mathfrak q$ is said to be $\mathfrak p$-primary.

By the Lasker--Noether Theorem (which Lasker proposed to call  the Noether--Dedekind Theorem)
every ideal $\mathfrak a$ 
has an irredundant primary decomposition into primary ideals
$   \mathfrak a=\mathfrak q_{1}\cap \cdots \cap \mathfrak q_{n}$ 
(for a proof see  van der Waerden's Algebra II 
\cite[Ch. 15]{vdW67}). The ideals $\mathfrak p_i=\sqrt {\mathfrak q_i}$ 
are the associated primes of $\mathfrak a$.

The decomposition is not unique. The simplest example is the ideal
$(x^2,xy)\subset K[x,y]$, which can be written as $(x)\cap (x^2,xy,y^2)$
but also as  $(x)\cap (x^2,y)$, and even $(x)\cap (x^2,y+\lambda x)$.
The number of components is always the same, and also the associated
prime ideals; in the example the ideals $(x)$ and $(x,y)$. According to 
\cite[footnote 10]{N21} Noether learned this example from K. Hentzelt.
This means that Emmy Noether occupied herself with primary ideals
already in 1913/14.
Macaulay \cite{M13,M16} has more complicated examples. Maybe Noether
knew about the manuscript for \cite{M13} through her father; the paper has
footnotes on p. 71 and p. 86, mentioning Max Noether, saying among others:
``I am also indebted to Professor Noether for kindly suggesting
 other alterations which I have carried out''. This sounds 
as a reaction to a (non-anonymous) referee report.

By  the definition of a primary ideal and the fact that the
associated prime ideal is finitely generated  one immediately obtains
\begin{prop}\label{M32}
If $\mathfrak q$ is a primary ideal and $\mathfrak p$ the associated prime ideal, then some
finite power of $\mathfrak p$ is contained in $\mathfrak q$.
\end{prop}

Noether \cite{N21} applies her  results to  polynomial ideals; in this paper she
still only
considers  complex coefficients. The  connection with  elimination theory
and ideal theory, say as described by Macaulay in his Tract \cite{M16},
 is given by 
the following special case of the Nullstellensatz.
\begin{prop}\label{prime-zero}
A prime ideal $\mathfrak p$ consists of all polynomials vanishing on  its  zero set. 
\end{prop} 

Conversely, the Nullstellensatz follows from this Proposition. 
Let $\mathfrak q_i$ be a primary ideal in the decomposition of $\mathfrak a$ with associated
prime ideal $\mathfrak p_i$. If $\mathfrak b$ is an ideal vanishing on the zero set of
 $\mathfrak a$, then it
vanishes on the zero set of $\mathfrak p_i$ and therefore  
 $\mathfrak b\subset \mathfrak p_i$  and $\mathfrak b^{k_i}\subset \mathfrak p_i^{k_i} 
\subset \mathfrak q_i$. Let $k$ be the maximum of the $k_i$, 
then $\mathfrak b^k\subset \bigcap \mathfrak q_i=\mathfrak a$.

Lasker stated Proposition \ref{prime-zero} in \cite{L}, adding that  it follows, say,  from the 
Nullstellensatz, but in the addendum \cite{L-err} he explained how to prove it
directly. It seems that Macaulay \cite{M16} did not notice this, as he criticises Lasker's 
proof of Proposition \ref{M32}, saying that Lasker 
first assumes the result and then proves it.

Macaulay and Lasker have a different definition of primary ideals, which makes 
the proof of Proposition \ref{M32} non-trivial.
Macaulay \cite{M16} defines 
a primary ideal  by the property that no product of two ideals
is contained in it without one of them contained in it or both containing 
its zero set. Hence if one does not contain the zero set the other is contained
in the ideal. 
By the Nullstellensatz Macaulay's definition is equivalent to Noether's.

Lasker's original definition was stated for  homogeneous ideals
in $S=K[x_1,\dots,x_n]$, making 
the statements about primary decomposition more complicated.
A primary ideal $\mathfrak q$ and the associated prime ideal $\mathfrak p$ 
occur both in his definition:  whenever $ab\in \mathfrak q$
and $a\notin \mathfrak p$ it follows that $b\in \mathfrak q$.
To make the zero set $C$ of $\mathfrak p$ an irreducible component of the zero set 
of $\mathfrak q$ it is required that the dimension of $\mathfrak q$ is at most that  of $\mathfrak p$.
According to Lasker
an algebraic set $C$ has dimension $m$ if it has a finite number of points
in common with $m$ general linear forms, and if
the forms in an ideal $\mathfrak a$ vanish on sets of dimension $m$, but not on
sets of higher dimension, then $\mathfrak a$ has dimension $m$. Actually
Lasker uses the quantity $m+1$, which he calls ``Mannigfaltigkeit'' and 
Macaulay \cite{M13} translates with manifoldness or dimensionality.
The value 
0 is allowed for dimensionality, and it occurs in \cite{L} in an essential way, 
although it is not defined what it means.

Lasker's approach to primary decomposition is as follows.
Let $C_1$, \dots, $C_j$ be the irreducible components of the zero set of $\mathfrak a$
of highest dimension and let $\mathfrak p_i$ be the prime ideal corresponding to $C_i$. 
Define $\mathfrak a_i$ as the set of all $f\in S$ such that $C_i$ is not a component of the zero set of the ideal quotient $\mathfrak a:(f)=\{g\in S\mid gf\in \mathfrak a\}$, so there exists a $\phi$
not vanishing on $C_i$ with $f\phi\in \mathfrak a$. Then $\mathfrak a_i$ is an ideal, whose zero set
only consists of $C_i$.  Furthermore $\mathfrak a_i$ is primary, for  if $ab\in \mathfrak a_i$  with  $a\notin \mathfrak p_i$,
then $ab\phi \in \mathfrak a$ for a $\phi$ not vanishing on $C_i$, and as $a\phi\notin \mathfrak p_i$, we have 
$b\in \mathfrak a_i$ by the definition of $\mathfrak a_i$.

The set $C_i$ is not a component of the zero set of the ideal quotient  
$\mathfrak a_j'=\mathfrak a:\mathfrak a_j$.
Let $\psi \in \mathfrak a_1'+\dots+\mathfrak a_j'$ be a form  which does not vanish
 on  any of the $C_i$.
Then we claim that $\mathfrak a=\mathfrak a_1\cap\dots\cap \mathfrak a_j\cap (\mathfrak a,\psi)$. 
If $f$ is an element of the
right hand side, then $f\in (\mathfrak a,\psi)$ so $f-g\psi \in \mathfrak a\subset \mathfrak a_i$, and as $f\in\mathfrak a_i$,
we get $g\psi\in \mathfrak a_i$; because $\mathfrak a_i$ is primary, and $\psi$ does not vanish
on $C_i$, we get $g\in \mathfrak a_i$  for all $i$. As $\psi \in \mathfrak a_1'+\dots+\mathfrak a_j'$ we find that 
$g\psi\in \mathfrak a$ and therefore $f\in \mathfrak a$. The dimension of $(\mathfrak a,\psi)$ is lower, and we can repeat the process.

In this proof essential use is made of Proposition \ref{prime-zero}.
Lasker proves it in \cite{L-err},
Macaulay in \cite[Section 31]{M16} and formulates it as follows:
there is only one
prime ideal with a given (irreducible) zero set, viz. the ideal consisting
of all polynomials vanishing on the zero set. 
In \cite[Section 32]{M16} he shows Proposition \ref{M32}. With primary decomposition
the Hilbert-Netto Theorem (Macaulay's name for the
Nullstellensatz) then follows.

Macaulay proves Propositions \ref{prime-zero} and \ref{M32} with classical methods of 
elimination theory, which 
are unfamiliar to the modern reader. The main ingredient is the so-called
$u$-resolvent. The goal is to describe the irreducible
components of the zero set of an ideal. Consider $\mathfrak a=(F_1,\dots,F_k)$, which
is as always supposed to be prepared by a general linear coordinate change.
Macaulay shortly writes  ``The solutions of $F_1=F_2=\dots=F_k=0$ are obtained
in the most useful way by introducing a general unknown $x$ standing for
$u_1x_1+u_2x_2+\dots+u_nx_n$, where $u_1,u_2,\dots,u_n$ are undetermined
coefficients''. This is known as the Liouville substitution \cite{Liou}, and its use is explained
in detail in Netto's book \cite{Netto-Alg}. One substitutes 
\[
x_1=\frac{x-u_2x_2-\dots-u_nx_n}{u_1}
\]
in the equations and multiplies with suitable powers of $u_1$ to make the
new equations $f_1=f_2=\dots=f_k=0$ in $x,x_2,\dots,x_n$ polynomial. The 
solutions of the first system determine those of the second and vice versa.
The complete resolvent $D_uD_u^{(1)}\cdots D_u^{(n-1)} (=F_u)$ of
$(f_1,\dots,f_k)$ obtained by eliminating $x_2,x_3,\dots,x_n$ 
in this order is called the complete $u$-resolvent of $(F_1,\dots,F_k)$.

The $u$-resolvent $F_u$ is a polynomial in $x,x_2,\dots,x_n,u_1,\dots,u_n$. As a
polynomial in $x$, when the $x_i$ and $u_i$ have specific values,
it splits in linear factors. Such a linear factor  of $D_u^{(r-1)}$ has the form
\[x-u_1\xi_1-u_r\xi_r-u_{r+1}x_{r+1}-\dots-u_nx_n\]
where  $\xi_1,\dots,\xi_r,x_{r+1},\dots,x_n$ is  a solution of $F_1=F_2=\dots=F_k=0$.
Macaulay calls a linear factor a true linear factor if $\xi_1,\dots,\xi_r$  are 
independent of $u_1,\dots,u_n$, that is, if it is linear in $x, u_1,\dots,u_n$.

\begin{ex}
Consider the ideal $(x_1^2+x_2^2-2,x_1^2-x_2^2)$.
Substituting $x=u_1x_1+u_2x_2$ and eliminating $x_2$ using the
Sylvester determinant  gives
$F_u = 4u_1^2 (x-u_2-u_1)(x-u_2+u_1)(x+u_2-u_1)(x+u_2+u_1)$.
From the factor $x-u_2-u_1$ one finds the solution $(x_1,x_2)=(1,1)$ by
substituting the values $(1,0)$ and $(0,1)$ for $(u_1,u_2)$.
Using the same substitution for the principal ideal $(x_1^2+x_2^2-2)$
gives $F_u= (x-u_1\sqrt{2-x_2^2}-u_2x_2) (x+u_1\sqrt{2-x_2^2}-u_2x_2)$, which is
indeed of the form $(x-u_1\xi_1-u_2x_2) (x-u_1\xi_2-u_2x_2)$. 
\end{ex}

Macaulay goes on to prove that the solution supplied by a factor which is not 
a true linear one is an embedded solution. In particular 
all the linear factors of  the first complete partial $u$-resolvent
are true linear factors.  According to Macaulay
Kronecker states without proof that all linear factors of $F_u$ are true
linear factors, while  König's proof contains an error, and Macaulay doubts
whether the statement is true.  In fact, Kronecker's claim is true for the
resultant form in the elimination theory of  Hentzelt-Noether \cite[Satz XIII]{HN}.

An irreducible factor $R_u$ of rank $r$ of $F_u$ having a  true linear factor
leads to a parametrisation of the corresponding irreducible component of the
solution set. One can take $x_{r+1},\dots,x_n$ arbitrary and for each set of
values there are solutions $x_{1i},\dots,x_{ri}$, $i=1,\dots,d$, with $d$ the degree
of the component. Therefore we can (formally) write
\[
R_u=A\,\prod_{i=1}^d (x-u_1x_{1i}-\dots-u_rx_{ri}-u_{r+1}x_{r+1}-\dots-u_nx_n)
\]
so 
\[
(R_u)_{x=u_1x_1+\dots+u_nx_n}=
A\,\prod_{i=1}^d (u_1(x_1-x_{1i})+\dots+u_r(x_r-x_{ri}))
\]
The last expression is independent of $u_{r+1},\dots,u_n$ and vanishes identically
at all points of the solutions set and at no other points, that is irrespective of
$u_1,\dots,u_r$. The coefficients of the monomials in the $u_i$ in 
$(R_u)_{x=u_1x_1+\dots+u_nx_n}$ are polynomials in the $x_i$ which all vanish
at all points of the solution set and do not all vanish at other points. This gives equations
for the solution set. We single out some of them.
The coefficient of $u_r^d$ is $\phi(x_r,x_{r+1},\dots,x_n)=A\,\prod (x_r-x_{ri})$.
The coefficient of $u_1u_r^{d-1}$ is $\phi\sum\frac{x_1-x_{1i}}{x_r-x_{ri}}$, which we 
write as $x_1\phi'-\phi_1$, where $\phi'$ is the derivative of $\phi$ w.r.t. $x_r$
and $\phi_1=\phi\sum\frac{x_{1i}}{x_r-x_{ri}}$. Similarly we have $x_2\phi'-\phi_2$,
\dots, $x_{r-1}\phi'-\phi x_{r-1}$. Outside $\phi'=0$ we have therefore the 
equations
\[ 
\phi=0,\quad x_i=\frac{\phi_i}{\phi'}, \; i=1,\dots,r-1\;.
\]

With these preparations we can give the proof of Proposition \ref{prime-zero}.
\begin{proof}[Macaulay's proof  of Proposition \ref{prime-zero}]
The zero set is irreducible, otherwise the complete $u$-resolvent
would contain at least two factors corresponding to different
irreducible components, contradicting that the ideal is prime.

Let $\mathfrak p=(F_1,\dots,F_k)$ be a prime ideal.
It will be sufficient to prove that  $F\in \mathfrak p$ for every polynomial $F$ 
that vanishes on the zero set of $\mathfrak p$.
The first complete partial $u$-resolvent of $\mathfrak p$ will be a power $R_u^m$ of an 
irreducible polynomial $R_u$ in $x,x_{r+1},\dots,x_n$.
The complete $u$-resolvent lies in the prime ideal $(f_1,\dots,f_k)$, and for dimension reasons the other factors do not vanish on the zero set  of $(f_1,\dots,f_k)$. Hence 
$R_u^m$ and therefore $R_u$ itself belongs to $(f_1,\dots,f_k)$.
This gives that $(R_u)_{x=u_1x_1+\dots+u_nx_n}\in (F_1,\dots,F_k)=P$, and the 
same holds for the polynomial coefficients of the monomials in the $u_i$.
In particular $\phi\in \mathfrak p$ and $\psi_i:=x_i\phi'-\phi_i\in \mathfrak p$, $i_1,\dots,r-1$.

Let now $F$ vanish on the zero set of $\mathfrak p$ and substitute 
$x_i=\phi_i/\phi'$, $i=1,\dots,r-1$; then $F$ becomes a rational function of 
$x_r,x_{r+1},\dots,x_n$  with denominator ${\phi'}^l$,  where $l$ is the degree of $F$.
This rational function vanishes for all points of the zero set of $\mathfrak p$ where $\phi'$ 
does not vanish and its numerator is therefore divisible by $\phi$. We
conclude that 
\[
\textstyle
{\phi'}^l F(\frac{\phi_1}{\phi'},\dots,\frac{\phi_{r-1}}{\phi'},x_r,\dots,x_n)
= G\phi
\]
for some polynomial $G$ in $x_r,\dots,x_n$. 
Therefore ${\phi'}^lF(x_1,\dots,
x_n)\in
(\psi_1, \dots, \psi_{r-1},\phi)\subset \mathfrak p$ and hence $F\in \mathfrak p$.
\end{proof}

Macaulay's proof of Proposition \ref{M32}
follows the same steps, but now we can only conclude that
$R_u^m\in(f_1,\dots,f_k)$. Taking suitable coefficients of 
$(R_u^m)_{x=u_1x_1+\dots+u_nx_n} $ we find that 
$\phi^m\in \mathfrak q=(F_1,\dots,F_k)$ and $\psi_i^m-G\phi \in  \mathfrak q$
so $\psi_i^{m^2}\in  Q$. If $F\in P$ with $P$ the associated prime ideal
then  we have just seen that $\phi'^lF\in (\psi_1,\dots,\psi_{r-1},\phi)$.
Therefore $(\phi'^lF)^{rm^2} \in (\psi_1^{m^2},\dots,  \psi_{r-1}^{m^2},\phi^{m^2})
\subset \mathfrak q$ and since $\mathfrak q$ is primary and no power of $\phi'$ 
is contained in $\mathfrak q$
we conclude that $F^{rm^2}\in \mathfrak q$.

Lasker gives in \cite{L-err} a totally different proof of Proposition \ref{prime-zero}, based on properties of the Hilbert polynomial 
\cite{Hilb90}.
Let   $\mathfrak a$ be a homogeneous ideal in $S=K[x_1,\dots,x_n]$.
Define the Hilbert function $H_\mathfrak a(\nu)=\dim (S/\mathfrak a)_\nu$, where $(S/\mathfrak a)_\nu$ is the degree
$\nu$ part of $S/\mathfrak a$. Lasker proves from scratch the now
familar properties of this function \cite[Kap. II]{L}.  
For large $\nu$ the function $H_\mathfrak a(\nu)$ is a polynomial in $\nu$.
Lasker shows that $H_\mathfrak a(\nu)=0$ for sufficiently large $\nu$ 
if the ideal $\mathfrak a$ has no zeros. This follows from the fact that
under this assumption all
monomials of large degree belong to $\mathfrak a$, that is, the weak projective Nullstellensatz 
\ref{WPNSS}. Lasker proves it using his version of the resultant of $n-1$ forms in
$n$ variables.
Furthermore, if $u$ is a form of degree $d$, which is not a zero divisor in $S/\mathfrak a$,
then  $H_{(\mathfrak a,u)}(\nu) = \Delta_d H_\mathfrak a(\nu)$, where $\Delta_d$
is defined by $\Delta_d f(\nu)=f(\nu)-f(\nu-d)$ for a function $f$.

In Lasker's version of Proposition 
\ref{prime-zero} (the zero set $C$ of a prime ideal is irreducible
and every form vanishing on $C$ belongs to $\mathfrak p$) irreducibility has to be proved.
%
Suppose that the zero set of the prime ideal $\mathfrak p$ consists only
of points. If $u\notin \mathfrak p$ is a linear form then $H_\mathfrak p(u)=\Delta_1 H_\mathfrak p$.
If $u$ does not vanish in any of the points, then 
$H_{(\mathfrak p,u)}=0=\Delta_1 H_\mathfrak p$ and $H_\mathfrak p$ is constant. If there would exist
a form $u\notin \mathfrak p$ vanishing in one of the points, then $H_{(\mathfrak p,u)}\neq0$,
so such a form does not exist.
We conclude that  every linear form vanishing in one of the points vanishes in
all, which is only possible if there is  only one point, showing  irreducibility. And every form
vanishing in the point belongs to $\mathfrak p$.

If $\mathfrak p$ has dimension 1 and $u$ does not contain any 1-dimensional component 
of the zero set, then again $H_{(\mathfrak p,u)}=\Delta_1 H_\mathfrak p$, so $H_\mathfrak p$ a linear
function of $\nu$. If there exists a form $u\notin \mathfrak p$, vanishing on a 1-dimensional
component $C$, then  $H_{(\mathfrak p,u)}$ is independent of $\nu$. The forms in
$(\mathfrak p,u)$ all vanish on $C$ and therefore are all contained in the 
prime ideal $\Pi$ of forms vanishing on $C$. The Hilbert polynomial of $\Pi$
is linear, so $H_{(\mathfrak p,u)}$ cannot be constant.  
This shows that every form vanishing on  the zero set of $\mathfrak p$ belongs to $\mathfrak p$.
Irreducibility also follows: suppose on the contrary that there 
exist $a,b$ with $ab$ vanishing while $a$ and $b$ do not
vanish in all points. Let $a$ contain a 1-dimensional component, but not all zeros.
Then $a\notin \mathfrak p$, so such a form cannot exist, and therefore $a$ vanishes in all points.

In this way the Proposition can be shown by induction.

The proof of Proposition \ref{M32} uses (or rather proves)
general properties of primary ideals,
and is in this way closer to the modern approach than Macaulay's.

\begin{proof}[Lasker's proof of Proposition \ref{M32}]
Let $F_1,\dots, F_h$ be a basis of $\mathfrak p$. Form $F=p_1F_1+\dots+p_hF_h$
with indetermined coefficients, of suitable degree.
The ideal quotient $\mathfrak q'=\mathfrak q:(F)=\{a\in S\mid aF\in \mathfrak q\}$
can be found despite the fact that $F$ has indetermined coefficients,
as the condition  $aF\in \mathfrak q$ leads in each 
each degree  to only linear equations, which can be solved  with
indetermined coefficients.
Put $\mathfrak q''=\mathfrak q':(F)$,  $\mathfrak q'''=\mathfrak q'':(F)$ and so on. In this sequence every ideal
is contained in the next, so there is a number $k$ with $\mathfrak q^{(k)}=\mathfrak q^{(k+1)}$,
by the Ascending Chain Condition (proved from the Basis Theorem by Lasker \cite[p. 56]{L}).
Every ideal $\mathfrak q^{(i)}$ is $\mathfrak p$-primary. Lasker proves this by doing the example $\mathfrak q'$. 
Let $a\notin \mathfrak p$ and 
$ab\in \mathfrak q'$, so $Fab\in \mathfrak q$ and because $\mathfrak q$ is primary,
$Fb\in \mathfrak q$, giving $b\in \mathfrak q'$. Also the 
dimension of $\mathfrak q'$ is at most that of $\mathfrak p$. Therefore $\mathfrak q'$ is $\mathfrak p$-primary
(according to Lasker's definition).

Moreover, according to Lasker's definition, if an ideal $\mathfrak q$ is $\mathfrak p$-primary then the zero set 
of $\mathfrak q$ contains the zero set of $\mathfrak p$ or $\mathfrak q$ is the whole ring: if $a\in \mathfrak q$, but $a\notin \mathfrak p$ 
and if $f$ is an arbitrary form, then $a f\in \mathfrak q$ so $f\in \mathfrak q$.
Now the above constructed $F$ is not a zero divisor on $S/\mathfrak q^{(k)}$, 
so by the properties of the Hilbert
polynomial the dimension of $(\mathfrak q^{(k)},F)$ should be less than that of $\mathfrak q^{(k)}$.
The conclusion is that $\mathfrak q^{(k)}$ is the whole ring, so $1\in \mathfrak q^{(k)}$.

As $\mathfrak q^{(k)}=\mathfrak q^{(k-1)}:(F)$, we get $F\in \mathfrak q^{(k-1)}$,
and then $F^2\in \mathfrak q^{(k-2)}$, until finally $F^k\in \mathfrak q$.
As the coefficients of the $p_i$ are indeterminates we conclude that $F_1^k$, 
$F_1^{k-1}F_2$, \dots, $F_h^k$ lie in $\mathfrak q$. Therefore $f^k\in \mathfrak q$ for any
form $f=q_1F_1+\dots+q_hF_h$ in $\mathfrak p$.
\end{proof}

\section{Modern algebra}
In Moderne Algebra II \cite{MAII} van der Waerden gives two proofs of the
Nullstellensatz, the first one using Rabinowitsch' trick and proving the weak version by elimination theory. The second proof is based on \cite{vdW26}
and belongs therefore to the
proofs before Rabinowitsch.  It proves Proposition \ref{prime-zero}, and using
Noether's definition of primary ideals the Nullstellensatz follows as above.
In later editions, in Algebra II \cite{vdW67}, the  elimination theory proof is removed, and the
weak version is shown with same type of ideas as in the second proof, but 
avoiding primary decomposition. This proof was first described in \cite{vdW54}.

Whereas Noether in \cite{N21} still considered complex coefficients in the application to
algebraic geometry, she takes in later papers always  an arbitrary field as base
field. In \cite{vdW26} this point is stressed by a footnote (Footnote 13), stating that the
new definition holds for unusual spaces, like those where the fourth
harmonic point always coincides with the third; this happens in characteristic two:
if the first three points on the
line are normalised to be $0,\infty$ and $1$, then the fourth harmonic has
coordinate $-1$.

Let $K$ be a field and $R=K[x_1,\dots,x_n]$ the polynomial ring in 
$n$ variables over $K$. Consider points in the affine space $\A^n(L)$
with coordinates in an algebraic extension $L$ of $K$. Besides such points
one has also to consider \lq undetermined\rq\ points, where the coordinates
are indeterminates or algebraic functions of parameters, that is elements in
 a transcendental extension $\Omega$ of $K$.

Let therefore $\Omega=K(\xi_1,\dots,\xi_n) $ be a field extension. The polynomials
$f\in R$ for which $f(\xi_1,\dots,\xi_n)=0$ form a prime ideal $\mathfrak p$ in $R$:
if \[f(\xi_1,\dots,\xi_n)g(\xi_1,\dots,\xi_n)=0\]
and $g(\xi_1,\dots,\xi_n)\neq0$, then $f(\xi_1,\dots,\xi_n)=0$, as a field
does not contain zero divisors. 
Van der Waerden \cite{vdW26}
gives a simple example: let $\xi_1,\dots,\xi_n$ be linear functions
of one indeterminate $t$ with coefficients in $K$: 
\[
\xi_i=\alpha_i+\beta_i t\;.
\]
Then $\mathfrak p$ consists of all polynomials vanishing on the line given 
by the above parametrisation.
This example is not contained in Moderne Algebra II \cite{MAII}, but  occurs again in Algebra II
\cite{vdW67}.

The field $\Omega$ is isomorphic 
to the quotient field $\Pi$ of $R/\mathfrak p$, in such a way that the
$\xi_1$ correspond to the $x_i$.
Conversely, for 
every prime ideal $\mathfrak p\neq 0$ there exists a field
$\Omega=K(\xi_1,\dots,\xi_n)$ such that $\mathfrak p$ consists of
all polynomials
$f\in R$ for which $f(\xi_1,\dots,\xi_n)=0$.
The point $(\xi_1,\dots,\xi_n)$ is the general zero
of $\mathfrak p$.

The dimension of $\mathfrak p$ is the transcendence degree of $\Omega$
over $K$. Let  $t_1,\dots,t_r$ be a transcendence
basis of $\Omega$, so $\Omega$ is an algebraic extension of
the field of rational functions $(t_1,\dots,t_r)$. 
Let $f_1,\dots,f_s$ be elements of the function field $\Omega$. For given values
$\tau_1,\dots,\tau_r$ of the argument one can solve and find values
$\vp_1,\dots,\vp_s$ in a suitable extension of $k$, but only those systems of values are
allowed for which all relation $F(f_1,\dots,f_s,t_1,\dots,t_r)=0$ also hold for the 
specific values, that is  $F(\vp_1,\dots,\vp_s,\tau_1,\dots,\tau_r)=0$. For example,
if $f_1=\sqrt t$, $f_2=-f_1$, then we have the equations $f_1^2=t$ and  $f_2^2=t$,
giving for $t=1$ the values $\vp_1=\pm 1$ and  $\vp_2=\pm 1$, but it is not allowed
to combine $\vp_1= 1$ with  $\vp_2= 1$, as this violates the relation $f_1+f_2=0$
\cite[\S 88]{MAII}. The existence of such systems is shown by adjoining the $f_i$
successively. The denominators in the resulting monic equations for the $f_i$
can be taken to depend only on the $t_j$. Let $V(t_1,\dots,t_r)$ be the lowest
common multiple of the denominators. Consider only parameter values for which
$V(\tau_1,\dots,\tau_r)\neq 0$. But also the converse is valid:
if a relation $F(\vp_1,\dots,\vp_s,\tau_1,\dots,\tau_r)=0$ holds for all 
regular systems ofvalues for the arguments and all admissable corresponding
function values, then also the relation  $F(f_1,\dots,f_s,t_1,\dots,t_r)=0$ 
holds in the function field \cite[\S 88]{MAII}. 

Every system of algebraic functions $\xi_1,\dots,\xi_n$ of $t_1,\dots,t_r$
can be specialised in the above way  to  $\xi_1',\dots,\xi_n'$ and this determines
a point  $\xi'$ in affine space over a suitable
algebraic extension of $k$. Let $V$ be the Zariski closure of these points, that is the zero 
set of all polynomials $F$ for which $F(\xi_1',\dots,\xi_n')=0$. This means that 
 $\xi_1,\dots,\xi_n$ determines the algebraic variety $V$ in parameter form
and its prime ideal $\mathfrak p$ has the general zero $(\xi_1,\dots,\xi_n)$. As every prime ideal
$\mathfrak p$ has a general zero $(\xi_1,\dots,\xi_n)$, where the $\xi_i$ are 
algebraic functions of parameters $t_1,\dots,t_r$,  Proposition \ref{prime-zero} follows:
every prime ideal is the  ideal of its zero set.
In particular, the only prime ideal without zeros is the whole ring.

As application van der Waerden proves first the generalisation
of Noether's fundamental theorem to  zero dimensional ideals in arbitrary dimension.
König proved the theorem for the case of a complete intersection \cite[p. 385]{Koe03}, and 
Macaulay observed that the general case easily follows using primary decomposition
\cite[p. 61]{M16}.

\begin{thm}
Let $\mathfrak a$ be an ideal in $R=K[x_1,\dots,x_n]$  
with finitely many zeros $P_i$ in
$\A^n( K)$, $K$ algebraically closed.  
For a zero $P=(\xi_1,\dots,\xi_n)$ let $\mathfrak m_P=
(x_1-\xi_1,\dots,x_n-\xi_n)$.
There is an integer $\rho$ depending only on $\mathfrak a$
such that $f \in  \mathfrak a+\mathfrak m_{P_i}^\rho$  for  
all $i$ implies $f\in \mathfrak a$.
\end{thm}  

\begin{proof}
Let $\mathfrak a=\bigcap_i \mathfrak q_i$ be the primary decomposition of $\mathfrak a$. The associated
prime ideal of $\mathfrak q_i$ is $\mathfrak m_{P_i}$. For each $i$ there is an exponent
$\rho_i$ such that $\mathfrak m_{P_i}^{\rho_i}\subset \mathfrak q_i$ and then
$\mathfrak q_i=\mathfrak a+ \mathfrak m_{P_i}^{\rho_i}$. With $\rho = \max \rho_i$
the condition in the theorem implies that $f\in \mathfrak q_i$ for all $i$
and therefore $f\in \mathfrak a$.
\end{proof}

Lasker generalised Noether's theorem \cite[Satz XXVII]{L} to what Macaulay calls  the 
Lasker-Noether Theorem \cite[p. 61]{M16}. He formulates it roughly as follows. 

\begin{thm}
If $\mathfrak a=(F_1,F_2,\dots,F_k)$ and $F$ can be written as 
$F= P_1F_1+P_2F_2+\dots+P_kF_k$ , where the $P_i$ are power series, then there 
exists  a
polynomial $\phi$  not vanishing at the origin such that $F\phi \in \mathfrak a$.
\end{thm}

It follows that $F$ lies in every primary component containing the origin.  
For a criterion that $F\in \mathfrak a$ it suffices to impose
the power series condition in a finite number of points.

According to van der Waerden
 both Lasker's and Macaulay's proofs are insufficient; 
he adds a note in proof that the gaps
in proof by Macaulay are filled in correspondence between them \cite{vdW26}. 
The proof still needs
convergence of the power series involved, a condition not necessary in 
the proof van der Waerden says to have. The easiest proof seems to be due to 
Krull \cite{KrHB}, and it is this proof  which Macaulay gives in \cite{M34},
and refers to as a  hitherto unpublished result. This makes it probable that 
Macaulay learned it from van der Waerden.

A different generalisation is due to Hentzelt, and elaborated by Hermann \cite{Hermann}.
We give it in the formulation of Krull \cite[Nr. 20]{Kr-enz}.

\begin{thm}
For every ideal $\mathfrak a=(F_1,\dots,F_k)$  in $R=K[x_1,\dots,x_n]$  exists an exponent
$\rho$ depending only on $n$ and $k$ and the degrees of the $F_i$ such that 
$F\in \mathfrak a$, if $F\in \mathfrak a+\mathfrak p_i^\rho$ for all associated prime ideals
$\mathfrak p_i$  of $\mathfrak a$.
\end{thm}
%
In this formulation the hard part is to establish that the bound only depends on the
stated quantities. To make clear that it is  a Nullstellensatz, the condition can be
formulated as
$F\in \mathfrak a R_i+(x_1-\xi_1,\dots,x_n-\xi_n)^\rho$, where 
$(\xi_1,\dots,\xi_n)$ is the general zero of the prime ideal $\mathfrak p_i$
and $R_i=K(\xi_1,\dots,\xi_n)[x_1,\dots,x_n]$.
Hentzelt originally formulated the condition for all (infinitely many) geometric zeros 
$(\xi_1,\dots,\xi_n)$ of $\mathfrak a$,
that $F\in \mathfrak a+(x_1-\xi_1,\dots,x_n-\xi_n)^\rho$.

A non-constructive proof, not establishing the degree bound, 
was given by van der Waerden \cite{vdW26}. It uses reduction to 
the zero dimensional case.  It is explained in \cite[\S 133]{vdW67}.

\appendix

\section{The resultant}\label{appendix}

Let $A$ be a unique factorisation domain. We are interested in the question
when two  binary forms $F(X,Y),G(X,Y)\in A[X,Y]$ have a common factor.

\begin{prop}\label{resprop}
The binary forms $F$ and $G$ in $A[X,Y]$
have a non-constant factor $H$ in common, 
if and only if
there exist forms $U$ and $V$ of degree less than $\deg F$, 
resp.~$\deg G$,
not both vanishing, such that 
$VF+UG=0$.
\end{prop}
 
\begin{proof}
Suppose $VF=-UG$.
All irreducible factors of $F$ have to occur in $UG$, and not all can occur in
$U$, because $\deg U < \deg F$; therefore $F$ and $G$ have a factor in
common. Conversely, given $H$ one finds a $U$ and a $V$
such that $F=-UH$ and $G=VH$, so the equation $VF+UG=0$ is satisfied,
with $\deg U<\deg F$ and $\deg V<\deg G$.
\end{proof}

Suppose $\deg F=m$ and $\deg G=n$ and
consider the free module $A[X,Y]_{n+m-1}$ of forms of degree $m+n-1$,

The existence of a relation  $VF+UG=0$ is equivalent to the fact that
the forms $X^{n-1}F$, $X^{n-2}YF$, \dots, $Y^{n-1}F$, 
$X^{m-1}G$, \dots, $Y^{m-1}G$
are linearly dependent in vector space 
$Q(A)[X,Y]_{n+m-1}$ of dimension $m+n$, where $Q(A)$ is the quotient field of $A$. 
We represent a form
 $c_0X^{n+m-1}+ \dots + c_{n+m-1}Y^{n+m-1}$
 by the row vector $(c_0,\dots,c_{n+m-1})$; multiplying with the column vector $\mathcal{X}=(X^{n+m-1},\dots,Y^{n+m-1})^t$ gives back the form. 
 
Put
\begin{align*}
F&=a_0X^m+a_1X^{m-1}Y+\dots+a_mY^m, \\
G&=b_0X^n+b_1X^{n-1}Y+\dots+b_nY^n.
\end{align*}
 Writing out the forms  $X^{n-1}f$,  \dots, $Y^{n-1}F$, 
$X^{m-1}G$, \dots, $Y^{m-1}g$ in the  basis 
$X^{n+m-1},\dots,Y^{n+m-1}$ leads in this way to a matrix equation  $S_{F,G} 
\mathcal{X} = \mathcal{F}$, with $\mathcal{F}=(X^{n-1}f,  \dots,Y^{m-1}g)^t$
and $S_{F,G}$ the Sylvester matrix
\[
S_{F,G}  =
\begin{pmatrix}
a_0 &  a_1  & a_2 & \dots &a_m &  \\
& a_0 &  a_1  & a_2 & \dots &a_m &  \\
&&\multispan 5\dotfill\hphantom{n}\\
&&&a_0 &  a_1  & a_2 & \dots &a_m   \\
b_0 &  b_1  & \dots &b_{n-1}&b_n &  \\
&b_0 &  b_1   & \dots &b_{n-1}&b_n   \\
&&\multispan 5\dotfill\hphantom{nn}\\
&&&b_0 &  b_1   & \dots &b_{n-1}&b_n   
\end{pmatrix}\,.
\]
The determinant  $R(F,G)$ of this matrix is called the \emph{resultant}
of $F$ and $G$.

From Proposition \ref{resprop} follows:

\begin{prop}\label{result}
The forms $G$ and $G$ have a non-constant factor $H$ in common, 
if and only if $R(F,H)=0\in A$.
\end{prop}

If the resultant does not vanish, the matrix equation $S_{F,G}   
\mathcal{X} = \mathcal{F}$ can be ``solved'' by inverting the matrix
$S_{F,G}  $. To avoid denominators it is better to multiply with the
adjugate matrix.

\begin{prop}\label{polxres-in-id}
For any two forms $F$ and $G$ of degrees $m$ and $n$ in $A[X,Y]$ and each
monomial $X^{k-1}Y^{n+m-k}$
there exist forms
$U$ and $V$  such that:
\[
VF+UG=X^{k-1}Y^{n+m-k}R(F,G).
\]
\end{prop}

The above results also apply to the inhomogeneous case, by putting $Y=1$. If the 
(homogeneous) resultant vanishes, there is a common factor, but this factor can be a power
of $Y$. This means in particular that $a_0=b_0=0$, so the degree of the inhomogeneous
polynomials $f(x)=F(X,1)$ and $g(x)=G(X,1)$ are lower than $m$ and $n$. 
This situation can occur if 
the coefficients are variables, which can be specialised to specific values. We 
call $m$ the formal degree of the polynomial $f$ \cite[\S 34]{vdW71Alg}.
The resultant $R(f,g)$ of two polynomials $f(x)$ and $g(x)$ of formal degrees 
$m$ and $n$ vanishes if and only if $f$ and $g$ have a non-constant factor in common
or in both polynoials the highest coefficient vanishes.

For inhomogeneous polynomials one obtains from Proposition \ref{polxres-in-id}:
\begin{prop}\label{res-in-id}
For any  
two polynomials $f$ and $g$ there exist  polynomials
$u$ and $v$ with $\deg u <\deg
f$ and $\deg v <\deg g$, such that:
\[
vf+ug=R(f,g).
\]
\end{prop}
The resultant is in fact a universal polynomial
in the $a_i$ and $b_j$ with integral coefficients and the above relation
also holds for indeterminate $a_i$ and $b_j$.

\small
\providecommand{\noopsort}[1]{}
\providecommand{\bysame}{\leavevmode\hbox to3em{\hrulefill}\thinspace}
\providecommand{\href}[2]{#2}

\end{document}